\documentclass[authoryear]{elsarticle}
\usepackage{amsmath,amssymb,amscd}
\usepackage{empheq}
\usepackage{bm,mathrsfs}
\usepackage{subeqnarray}
\usepackage{natbib}
\newtheorem{theorem}{Proposition}
\newproof{proof}{Proof}

\begin{document}
\title{Almost Sure Uniform Convergence of a Random Gaussian Field Conditioned on a Large Linear Form to a Non Random Profile}
\author{Philippe Mounaix}
\ead{philippe.mounaix@cpht.polytechnique.fr}
\address{Centre de Physique Th\'eorique, Ecole
Polytechnique, CNRS, Universit\'e Paris-Saclay, F-91128 Palaiseau, France.}
\date{\today}
\begin{abstract}
We investigate the realizations of a random Gaussian field on a finite domain of ${\mathbb R}^d$ in the limit where a given linear functional of the field is large. We prove that if its variance is bounded, the field converges uniformly and almost surely to a non random profile depending only on the covariance and the considered linear functional of the field. This is a significant improvement of the weaker $L^2$-convergence in probability previously obtained in the case of conditioning on a large quadratic functional.
\end{abstract}
\maketitle
%
%
\section{Introduction}\label{sec1}
This note concerns the quasi-deterministic properties of a Gaussian random field in the limit where some linear functional of the field is large. More concretely, consider the quite common situation of a random input signal modeled by a Gaussian noise and, as an output, a linear functional of this noise (e.g., a weighted average, a convolution with a transfer function, a component of a multipole moment, etc.). The question is: is there a concentration of the input signal realizations as the output value gets large and, if yes, into what subset of realizations and in what mathematically rigorous sense does the concentration occur? As we will see in this note, the answer reveals a significantly stronger concentration than the one previously obtained in the case of conditioning on a large quadratic functional by\ \citet{MC} and\ \citet{M2015}.

Let $\varphi$ be a random Gaussian field on a finite domain of ${\mathbb R}^d$ with realizations in a separable Hilbert space $\mathscr{H}$. Generalizing the work in \ \citet{MC} on the concentration of Gaussian fields constrained by a large $L^2$-norm, it was shown in\ \citet{M2015} that in the limit where a given (real) quadratic functional of $\varphi$ gets large, $\varphi$ concentrates onto a finite, low-dimensional, subspace $\mathfrak{H}\subset\mathscr{H}$. More specifically, writing $\overline{\varphi}$ the projection of $\varphi$ onto $\mathfrak{H}$, and $\mathbb{P}_u$ the conditional probability knowing that the quadratic functional of $\varphi$ is greater than some $u\in\mathbb{R}$, it was proved in\ \citet{M2015} that for every $\varepsilon >0$,
\begin{equation}\label{oldresult}
\lim_{u\rightarrow +\infty}\mathbb{P}_u\left(\left\|\frac{\varphi}{\|\varphi\|_2}
-\frac{\overline{\varphi}}{\|\overline{\varphi}\|_2}\right\|_2\le\varepsilon\right) = 1.
\end{equation}
The subspace $\mathfrak{H}$ in which $\overline{\varphi}$ lives is entirely determined by both the covariance operator of $\varphi$ and the considered quadratic functional. In particular, if $\mathfrak{H}$ is one-dimensional, $\overline{\varphi}$ is non random to within an overall multiplicative factor, (abbreviated in the following as `$\overline{\varphi}$ has a non random profile').

In this note we prove that when $\varphi$ is conditioned on a large linear (instead of quadratic) functional, it concentrates onto a one-dimensional subspace of $\mathscr{H}$ (i.e., $\overline{\varphi}$ has a non random profile) and the $L^2$-convergence in probability, as expressed by Eq.\ (\ref{oldresult}), turns into a considerably stronger, almost sure and uniform, convergence.
%
%
\section{Definitions and notation}\label{sec2}
Let $\Lambda$ be a bounded subset of $\mathbb{R}^d$, $\mathbb{F}=\mathbb{C}$ (or $\mathbb{R}$) the field of scalars, and $\mathscr{H}=L^2(\Lambda)\otimes\mathbb{F}^N$ the set of $N$-tuples $\lbrace\phi_1(x),\phi_2(x),\cdots ,\phi_N(x)\rbrace$ of complex (or real) square-integrable functions on $\Lambda$. The inner product of $\mathscr{H}$ is defined by (using Dirac's bracket notation)
\begin{equation}\label{scalarprod}
\langle\psi\vert\phi\rangle = \sum_{m=1}^N\int_\Lambda\psi_m(x)^\ast \phi_m(x)\, d^dx.
\end{equation}
Let $\hat{C}$ be a positive, trace class, operator on $\mathscr{H}$ defining the covariance operator of a Gaussian measure on $\mathscr{H}$ with support $\mathscr{D}(\hat{C}^{-1/2})\subset\mathscr{H}$, the domain of $\hat{C}^{-1/2}$. We consider the class of zero-mean random Gaussian fields $\vert\varphi\rangle$ which can be written as a generalized Karhunen-Loeve expansion\ \citep{KLE}
\begin{equation}\label{eq1}
\vert\varphi\rangle =\sum_{n=1}^{+\infty}t_n\hat{C}^{1/2}\vert\nu_n\rangle ,
\end{equation}
where $\lbrace\vert\nu_n\rangle\rbrace$ ($n=1,2,\cdots$) is an orthonormal basis\footnote{such a basis exists as $\mathscr{H}$ is a separable Hilbert space} of $\mathscr{H}$, and the $t_n\in\mathbb{F}$ are i.i.d. Gaussian random variables with either $\langle t_n\rangle =\langle t_n^2\rangle =0$ and $\langle\vert t_n\vert^2\rangle =1$ if $\mathbb{F}=\mathbb{C}$, or $\langle t_n\rangle =0$ and $\langle t_n^2\rangle =1$ if $\mathbb{F}=\mathbb{R}$. From a physical point of view, this requirement is not very restrictive as every centered Gaussian field with a continuous correlation function has an expansion of this form\ \citep{KLE}.

Let $\langle T\vert$ be a linear functional over a subspace of $\mathscr{H}$ such that $\langle T\vert\hat{C}^{1/2}$ is in the (continuous) dual of $\mathscr{H}$. By the Riesz representation theorem, one has $\hat{C}^{1/2}\vert T\rangle\in\mathscr{H}$ and $\|\hat{C}^{1/2}\vert T\rangle\|_2^2 =\langle T\vert\hat{C}\vert T\rangle <+\infty$. As we want to condition $\vert\varphi\rangle$ on a large $\langle T\vert\varphi\rangle$, first we must make sure that $\langle T\vert\varphi\rangle$ does exist, at least almost surely. This is done by the proposition,
\begin{theorem}
$\vert\langle T\vert\varphi\rangle\vert <+\infty$ almost surely.
\end{theorem}
\begin{proof}
From the expression\ (\ref{eq1}) of $\vert\varphi\rangle$ and $\mathscr{D}(\hat{C}^{-1/2})$ being the support of the Gaussian measure defined by the covariance operator $\hat{C}$, one gets
\begin{equation}\label{sum1}
\|\hat{C}^{-1/2}\vert\varphi\rangle\|_2^2=\sum_{n=1}^{+\infty}\vert t_n\vert^2 <+\infty,
\end{equation}
almost surely, and by Cauchy-Schwarz inequality and $\langle T\vert\hat{C}\vert T\rangle <+\infty$, one has
\begin{eqnarray}\label{exists}
\vert\langle T\vert\varphi\rangle\vert &=&\left\vert\sum_{n=1}^{+\infty}
t_n\langle T\vert\hat{C}^{1/2}\vert\nu_n\rangle\right\vert \nonumber \\
&\le&\left(\sum_{n=1}^{+\infty}\vert\langle T\vert\hat{C}^{1/2}\vert\nu_n\rangle\vert^2\right)^{1/2}\, 
\left(\sum_{n=1}^{+\infty}\vert t_n\vert^2\right)^{1/2} \nonumber \\
&=&\sqrt{\langle T\vert\hat{C}\vert T\rangle}\, \left(\sum_{n=1}^{+\infty}\vert t_n\vert^2\right)^{1/2}
<+\infty ,
\end{eqnarray}
almost surely.
\end{proof}
In order to deal with both cases $\mathbb{F}=\mathbb{C}$ and $\mathbb{F}=\mathbb{R}$ at once, we use the notation $\vert s\vert_{\mathbb{F}}=\vert s\vert$ (resp. $\vert s\vert_{\mathbb{F}}=s$) for $s\in\mathbb{F}$ and $\mathbb{F}=\mathbb{C}$  (resp. $\mathbb{F}=\mathbb{R}$). Since $\langle T\vert\varphi\rangle$ exists almost surely (see Proposition 1), we can now define $\vert\varphi_u\rangle$ the conditional random field $\vert\varphi\rangle$ knowing that $\vert\langle T\vert\varphi\rangle\vert_{\mathbb{F}}\ge u$. By choosing the orthonormal basis $\lbrace\vert\nu_n\rangle\rbrace$ in Eq.\ (\ref{eq1}) such that
\begin{equation}\label{eq2}
\vert\nu_1\rangle =\frac{\hat{C}^{1/2}\vert T\rangle}{\sqrt{\langle T\vert\hat{C}\vert T\rangle}},
\end{equation}
one gets
\begin{equation}\label{eq3}
t_1=\langle\nu_1\vert\hat{C}^{-1/2}\vert\varphi\rangle =
\frac{\langle T\vert\varphi\rangle}{\sqrt{\langle T\vert\hat{C}\vert T\rangle}},
\end{equation}
and from\ (\ref{eq3}) and the statistical independence of the $t_n$ it follows that the conditional field $\vert\varphi_u\rangle$ can also be written as the Karhunen-Loeve expansion\ (\ref{eq1}) in which $t_1$ is replaced with the corresponding conditional random variable $t_u$ knowing that $\vert t_1\vert_{\mathbb{F}}\ge u\langle T\vert\hat{C}\vert T\rangle^{-1/2}$. Let $\rho\in\lbrack 0,+\infty)$ be the random variable, independent of the $t_n$ with $n\ge 2$, defined by $\vert t_u\vert_{\mathbb{F}}=(\rho +u^2\langle T\vert\hat{C}\vert T\rangle^{-1})^{1/2}$. In the following we will not need the full expression of the probability distribution\footnote{If $\mathbb{F}=\mathbb{C}$, $\rho$ is an exponential random variable with parameter $1$ independent of $u$. If $\mathbb{F}=\mathbb{R}$, the PDF of $\rho$ depends on $u$ but is supported on $\lbrack 0,+\infty)$ independent of $u$ and tends pointwise to an exponential distribution with parameter $1/2$, independent of $u$, as $u\rightarrow +\infty$.} of $\rho$ but only the fact that its support $\lbrack 0,+\infty)$ is independent of $u$, which can easily be checked from the distribution of $t_u$, as obtained from the Gaussian distribution of $t_1$, and the change of variable $t_u\rightarrow\rho$. The point is that we can fix $\rho\in\lbrack 0,+\infty)$ regardless of the value of $u$, making it possible to let $u\to +\infty$ at fixed $\rho\in\lbrack 0,+\infty)$.

Finally, we write $\|\cdot\|_{\infty}$ the sup-norm defined by
$$
\|\psi\|_{\infty}=\max_{1\le m\le N}\sup_{x\in\Lambda}\vert\langle x,m\vert\psi\rangle\vert ,
$$
and $\mathbb{P}^{(n\ge 2)}$ the Gaussian product probability measure of the $t_n$ with $n\ge 2$.
%
%
\section{Almost sure uniform convergence of $\bm{\varphi}$ in the large $\bm{\vert\langle T\vert\varphi\rangle\vert_{\mathbb{F}}}$ limit}\label{sec3}
We are interested in the behavior of $\vert\varphi_u\rangle$ in the large $u$ limit. Let $\vert\overline{\varphi}\rangle ={\rm e}^{i\theta}\hat{C}\vert T\rangle$ where $\theta =\arg t_u$ ($\theta$ is uniformly distributed over $\lbrack 0,2\pi)$ if $\mathbb{F}=\mathbb{C}$, and $\theta=0$ if $\mathbb{F}=\mathbb{R}$ and $u>0$). Note that $\overline{\varphi}$ has a non random profile as $\hat{C}\vert T\rangle$ is non random and $\vert\overline{\varphi}\rangle$ lives in the one-dimensional subspace of $\mathscr{H}$ spanned by $\hat{C}\vert T\rangle$. The main result of this note is the following proposition:
\begin{theorem}
If $\langle x,m\vert\hat{C}\vert x,m\rangle$ is bounded on $x\in\Lambda$ and $1\le m\le N$, then for all (fixed) $\rho$ and $\theta$ one has
\begin{equation}\label{prop2}
\mathbb{P}^{(n\ge 2)}\left(\lim_{u\rightarrow +\infty}
\left\|\frac{\varphi_u}{\|\varphi_u\|_2}-\frac{\overline{\varphi}}{\|\overline{\varphi}\|_2}\right\|_{\infty}
=0\right)=1.
\end{equation}
\end{theorem}
\begin{proof}
First, we prove some useful inequalities. Let $a_n$ ($n=1,2,\dots$) a sequence of numbers such that $\sum_{n=1}^{+\infty}\vert a_n\vert^2 <+\infty$.  Assume that $\langle x,m\vert\hat{C}\vert x,m\rangle$ is bounded on $x\in\Lambda$, $1\le m\le N$, and write $\max_{1\le m\le N}\sup_{x\in\Lambda}\langle x,m\vert\hat{C}\vert x,m\rangle =A^2$. By Cauchy-Schwarz inequality one has
\begin{eqnarray}\label{ineq1}
&&\left\vert\sum_{n=1}^{+\infty}a_n\langle x,m\vert\hat{C}^{1/2}\vert\nu_n\rangle\right\vert\le
\left(\sum_{n=1}^{+\infty}\vert a_n\vert^2\right)^{1/2}
\left(\sum_{n=1}^{+\infty}\vert\langle x,m\vert\hat{C}^{1/2}\vert\nu_n\rangle\vert^2\right)^{1/2} \nonumber \\
&&=\sqrt{\langle x,m\vert\hat{C}\vert x,m\rangle}\left(\sum_{n=1}^{+\infty}\vert a_n\vert^2\right)^{1/2}
\le A\left(\sum_{n=1}^{+\infty}\vert a_n\vert^2\right)^{1/2}.
\end{eqnarray}
Integrating the square of\ (\ref{ineq1}) with $a_1=0$ over $x\in\Lambda$ and summing over $m$ from $1$ to $N$ yields
\begin{equation}\label{ineq2}
\left\|\sum_{n=2}^{+\infty}a_n\hat{C}^{1/2}\vert\nu_n\rangle\right\|_2 \le
A\sqrt{N\vert\Lambda\vert}\left(\sum_{n=2}^{+\infty}\vert a_n\vert^2\right)^{1/2}.
\end{equation}
From Minkowski inequality,\ (\ref{eq2}), and\ (\ref{ineq2}) one gets
\begin{eqnarray}\label{ineq3}
&&\left\|\sum_{n=1}^{+\infty}a_n\hat{C}^{1/2}\vert\nu_n\rangle\right\|_2 \le
\left\| a_1\hat{C}^{1/2}\vert\nu_1\rangle\right\|_2 +
\left\|\sum_{n=2}^{+\infty}a_n\hat{C}^{1/2}\vert\nu_n\rangle\right\|_2 \nonumber \\
&&=\vert a_1\vert\sqrt{\frac{\langle T\vert\hat{C}^2\vert T\rangle}{\langle T\vert\hat{C}\vert T\rangle}} +
\left\|\sum_{n=2}^{+\infty}a_n\hat{C}^{1/2}\vert\nu_n\rangle\right\|_2 \nonumber \\
&&\le\vert a_1\vert\sqrt{\frac{\langle T\vert\hat{C}^2\vert T\rangle}{\langle T\vert\hat{C}\vert T\rangle}} +
A\sqrt{N\vert\Lambda\vert}\left(\sum_{n=2}^{+\infty}\vert a_n\vert^2\right)^{1/2},
\end{eqnarray}
and, for $\vert a_1\vert$ large enough,
\begin{eqnarray}\label{ineq4}
&&\left\|\sum_{n=1}^{+\infty}a_n\hat{C}^{1/2}\vert\nu_n\rangle\right\|_2 \ge
\left\vert\, \left\| a_1\hat{C}^{1/2}\vert\nu_1\rangle\right\|_2 -
\left\|\sum_{n=2}^{+\infty}a_n\hat{C}^{1/2}\vert\nu_n\rangle\right\|_2\, \right\vert \nonumber \\
&&=\vert a_1\vert\sqrt{\frac{\langle T\vert\hat{C}^2\vert T\rangle}{\langle T\vert\hat{C}\vert T\rangle}} -
\left\|\sum_{n=2}^{+\infty}a_n\hat{C}^{1/2}\vert\nu_n\rangle\right\|_2
\ \ \ ({\rm for}\ \vert a_1\vert\ {\rm large\ enough})\nonumber \\
&&\ge\vert a_1\vert\sqrt{\frac{\langle T\vert\hat{C}^2\vert T\rangle}{\langle T\vert\hat{C}\vert T\rangle}} -
A\sqrt{N\vert\Lambda\vert}\left(\sum_{n=2}^{+\infty}\vert a_n\vert^2\right)^{1/2}.
\end{eqnarray}
Now, we use these inequalities to prove\ (\ref{prop2}). Since $\sum_{n=1}^{+\infty}\vert t_n\vert^2 <+\infty$ almost surely (see Eq.\ (\ref{sum1})), we can take $a_n=t_n/\|\varphi_u\|_2$ ($n\ge 2$) in Eq.\ (\ref{ineq1}) for almost all the realizations of $\lbrace t_n\rbrace$. Taking then
$$
a_1=\frac{t_u}{\|\varphi_u\|_2}-{\rm e}^{i\theta}\sqrt{\frac{\langle T\vert\hat{C}\vert T\rangle}{\langle T\vert\hat{C}^2\vert T\rangle}},
$$
and using the Karhunen-Loeve expansion of $\vert\varphi_u\rangle$ and Eq.\ (\ref{eq2}) to write $\vert\overline{\varphi}\rangle$ as $\vert\overline{\varphi}\rangle ={\rm e}^{i\theta}\sqrt{\langle T\vert\hat{C}\vert T\rangle}\, \hat{C}^{1/2}\vert\nu_1\rangle$, one gets, for all $t_u$ and almost all the realizations of $\lbrace t_n\rbrace$ ($n\ge 2$),
\begin{equation}\label{estimate0}
\left\|\frac{\varphi_u}{\|\varphi_u\|_2}-\frac{\overline{\varphi}}{\|\overline{\varphi}\|_2}\right\|_{\infty} \le
A\left(\left\vert\frac{t_u}{\|\varphi_u\|_2}-{\rm e}^{i\theta}\sqrt{\frac{\langle T\vert\hat{C}\vert T\rangle}{\langle T\vert\hat{C}^2\vert T\rangle}}\right\vert^2 +\frac{1}{\|\varphi_u\|_2^2}\sum_{n=2}^{+\infty}\vert t_n\vert^2\right)^{1/2}
\end{equation}
where we have used that the bound on the right-hand side of\ (\ref{ineq1}) does not depend on $x$ and $m$. Write $B=\sqrt{\langle T\vert\hat{C}\vert T\rangle /\langle T\vert\hat{C}^2\vert T\rangle}$ and $D=AB\sqrt{N\vert\Lambda\vert}$. From Eqs.\ (\ref{ineq3}) and\ (\ref{ineq4}) with $a_1=t_u$ and $a_n=t_n$ ($n\ge 2$), one gets the estimates
\begin{equation}\label{estimate1}
\frac{B}{1+D\vert t_u\vert^{-1}(\sum_{n=2}^{+\infty}\vert t_n\vert^2)^{1/2}}
\le \frac{\vert t_u\vert}{\|\varphi_u\|_2}\le
\frac{B}{1-D\vert t_u\vert^{-1}(\sum_{n=2}^{+\infty}\vert t_n\vert^2)^{1/2}},
\end{equation}
and
\begin{equation}\label{estimate2}
\frac{1}{\|\varphi_u\|_2^2}\sum_{n=2}^{+\infty}\vert t_n\vert^2 \le
\left(\frac{B\, (\sum_{n=2}^{+\infty}\vert t_n\vert^2)^{1/2}}
{\vert t_u\vert-D\, (\sum_{n=2}^{+\infty}\vert t_n\vert^2)^{1/2}}\right)^2,
\end{equation}
valid for almost all the realizations of $\lbrace t_n\rbrace$ ($n\ge 2$) and all $t_u$ with $u$ large enough. Now, it follows from\ (\ref{estimate1}),\ (\ref{estimate2}), and the expression of $B$, that for all $\rho$ and $\theta$, and almost all the realizations of $\lbrace t_n\rbrace$ ($n\ge 2$)
\begin{equation}\label{limit1}
\lim_{u\rightarrow +\infty}\frac{t_u}{\|\varphi_u\|_2}=
{\rm e}^{i\theta}\sqrt{\frac{\langle T\vert\hat{C}\vert T\rangle}{\langle T\vert\hat{C}^2\vert T\rangle}},
\end{equation}
and
\begin{equation}\label{limit2}
\lim_{u\rightarrow +\infty}\frac{1}{\|\varphi_u\|_2^2}\sum_{n=2}^{+\infty}\vert t_n\vert^2
=0,
\end{equation}
which, once injected onto the right-hand side of\ (\ref{estimate0}), yields
\begin{equation}\label{limit3}
\lim_{u\rightarrow +\infty}
\left\|\frac{\varphi_u}{\|\varphi_u\|_2}-\frac{\overline{\varphi}}{\|\overline{\varphi}\|_2}\right\|_{\infty}
=0,
\end{equation}
valid for all $\rho$ and $\theta$, and almost all the realizations of $\lbrace t_n\rbrace$ ($n\ge 2$), which completes the proof of Proposition 2.
\end{proof}
Convergence of $\varphi/\|\varphi\|_2$ to $\overline{\varphi}/\|\overline{\varphi}\|_2$ could also have been obtained from the results in\ \citet{M2015} on the concentration properties of a Gaussian field conditioned on a large quadratic functional $\langle\varphi\vert\hat{O}\vert\varphi\rangle$ by taking $\hat{O}=\vert T\rangle\langle T\vert$. But, as mentioned in the introduction, the convergence would have been much weaker as can be seen by comparing Eqs.\ (\ref{oldresult}) and\ (\ref{prop2}). The price to pay to get the stronger convergence\ (\ref{prop2}) is the mild extra condition that the variance $\langle x,m\vert\hat{C}\vert x,m\rangle$ be bounded on $x\in\Lambda$ and $1\le m\le N$. This is not required by\ (\ref{oldresult}) which needs the conditions imposed on $\hat{C}$ and $\langle T\vert$ in Sec.\ \ref{sec2} only\ \citep[for details, see][]{M2015}.
%
%
\section{Application to a Gaussian field with a large  $\bm{n}$th derivative}\label{sec4}
Take $\mathscr{H}=L^2(\Lambda)\otimes\mathbb{C}$ (complex scalar fields) with $\Lambda$ a closed subset of $\mathbb{R}$. Write $C(x,y)=\langle x\vert\hat{C}\vert y\rangle$. One has the following proposition,
\begin{theorem}
If $C(x,x)$ is bounded on $x\in\Lambda$ and $\partial^{2n}C(x,y)/\partial x^n\partial y^n \vert_{x=y=x_0}$ exists at some given $x_0\in\Lambda$ (with $n\ge 0$ an integer), then $\varphi^{(n)}(x_0)$ exists almost surely and
\begin{equation}\label{prop3}
\frac{\varphi(x)}{\|\varphi\|_2}\, 
\longrightarrow\, 
{\rm e}^{i\theta}
\frac{\partial^n C(x,x_0)/\partial x_0^n}
{\|\partial^n C(\cdot,x_0)/\partial x_0^n\|_2},
\end{equation}
in sup-norm and almost surely as $\vert\varphi^{(n)}(x_0)\vert\rightarrow +\infty$,
where $\theta$ is a random phase uniformly distributed over $\lbrack 0,2\pi)$.
\end{theorem}
\begin{proof}
The proof is a direct application of Propositions 1 and 2 taking for $\langle T\vert$ the $n$th derivative (in the sense of distributions) of the Dirac mass at $x_0$.
\end{proof}
In the case of a real scalar field and $n=0$, this result has been known for long\ \citep[see e.g.][Sects. 6.7 and 6.8]{Adler}, but for a smaller class of smoother fields, with a twice derivable correlation function at $x=x_0$, and with a much weaker pointwise convergence in law.
\section*{Acknowledgments}
The author warmly thanks Pierre Collet and Satya N. Majumdar for interest as well as for useful discussions on related subjects.
%
%

%
%
\end{document}